\PassOptionsToPackage{compatibility=false}{caption}
\documentclass[letterpaper, 10 pt, conference]{ieeeconf}
\IEEEoverridecommandlockouts   
\usepackage{amsthm}
\usepackage[utf8]{inputenc}
\usepackage[american]{babel}
\usepackage{csquotes}
\usepackage[letterpaper,left=54pt,bottom=54pt,right=54pt,top=54pt]{geometry}

\expandafter\let\csname equation*\endcsname\relax
\expandafter\let\csname endequation*\endcsname\relax
\usepackage{mathtools}
\usepackage{amssymb,amsfonts,stmaryrd}
\usepackage{mathrsfs}  
\usepackage{tikz}
\usetikzlibrary{cd}
\usepackage[unicode=true, pdfusetitle, colorlinks=true, allcolors=blue, bookmarks=false]{hyperref}
\usepackage{graphicx,subcaption}
\usepackage[colorinlistoftodos]{todonotes} 

\theoremstyle{plain}
\newtheorem{theorem}{Theorem}
\newtheorem*{theorem*}{Theorem}

\newtheorem*{lemma*}{Lemma}
\newtheorem{proposition}[theorem]{Proposition}
\newtheorem*{proposition*}{Proposition}
  
\newtheorem*{corollary*}{Corollary}
\theoremstyle{definition}
\newtheorem{definition}{Definition}
\newtheorem*{definition*}{Definition}
\newtheorem{example}{Example}
\newtheorem*{example*}{Example}
\theoremstyle{remark}
\newtheorem{remark}{Remark}
\newtheorem*{remark*}{Remark}

\newtheorem*{conjecture*}{Conjecture}

\newtheorem*{problem*}{Problem}

\newcommand*{\RR}{\mathbb{R}}
\let\R\RR

\newcommand*{\dd}{\mathrm{d}}

\newcommand*{\contr}[1]{\iota_{#1}}
\newcommand*{\liedv}[1]{\mathcal{L}_{#1}}
\newcommand*{\Reeb}{\mathcal{R}}

\DeclareMathOperator{\Leg}{Leg}

\newcommand{\sode}{\Gamma}
\newcommand{\curve}{\sigma}

\newcommand{\parder}[2]{\frac{\partial #1}{\partial #2}}

\newcommand{\parderr}[3]{\frac{\partial^2 #1}{\partial #2\partial #3}}

\let\hat\widehat
\let\oldemph\emph
\let\myemph\emph

\mathtoolsset{showonlyrefs}

\title{\LARGE \bf Nonsmooth Herglotz variational principle}
\author{Asier L\'opez-Gord\'on$^{1}$, Leonardo Colombo$^{2}$, and  Manuel de Le\'on$^{3}$%
\thanks{$^1$A.~L.-G.~(asier.lopez@icmat.es) is with Instituto de Ciencias Matemáticas (ICMAT-CSIC), C/ Nicol\'as Cabrera, 13-15, 28049 Madrid, Spain.}
\thanks{$^2$L.~C.~(leonardo.colombo@car.upm-csic.es) is with Centro de Automática y Robótica (CAR-CSIC), Ctra. M300 Campo Real, Km 0, 200, Arganda del Rey - 28500 Madrid, Spain.}
\thanks{$^3$M.~d.~L.~(mdeleon@icmat.es) is with Instituto de Ciencias Matemáticas (ICMAT-CSIC) and Real Academia de Ciencias, Madrid, Spain.}
\thanks{The authors acknowledge financial support from Grant PID2019-106715GB-C21 funded by MCIN/AEI/ 10.13039/501100011033. Manuel de León and Asier López-Gordón also recieved support from the Grant CEX2019-000904-S funded by MCIN/AEI/ 10.13039/501100011033. Asier López-Gordón would also like to thank MCIN for the predoctoral contract PRE2020-093814.
}
}

\begin{document}
\newgeometry{left=48pt,bottom=43pt,right=48pt,top=60pt}

\maketitle
\thispagestyle{empty}
\pagestyle{empty}

\begin{abstract}
In this paper, the theory of smooth action-dependent Lagrangian mechanics (also known as contact Lagrangians) is extended to a non-smooth context appropriate for collision problems. In particular, we develop a Herglotz variational principle for non-smooth action-dependent Lagrangians which leads to the preservation of energy and momentum at impacts. By defining appropriately a Legendre transform, we can obtain the Hamilton equations of motion for the corresponding non-smooth Hamiltonian system. We apply the result to a billiard problem in the presence of dissipation.  
\end{abstract}


\section{Introduction}
Lagrangian and Hamiltonian systems are well-known for representing a large class of mechanical systems which are conservative.
However, many systems of interest in physics and engineering exhibit a dissipative, rather than a conservative, behaviour. As it is well-known, the equations of motion of a Lagrangian system (i.e., the Euler-Lagrange equations) can be obtained variationally from Hamilton’s principle of least action. A dissipative term may be added to the equations of motion by means of considering an external force besides the Lagrangian or the Hamiltonian function \cite{Godbillon1969,deLeon1989, deLeon2021a,Lopez-Gordon2021}.  
Alternatively, in order to describe mechanical systems with a damping term one can make use of a generalization of Hamilton’s Principle: the so-called Herglotz Principle. Essentially, one considers Lagrangian functions that can depend explicitly on the action (the so-called action-dependent Lagrangians) and, taking variations of the action, one obtains equations of motion, called Herglotz equations \cite{Herglotz1930}, that are like ordinary Euler-Lagrange equations with an extra term accounting for dissipation. In recent years, action-dependent Lagrangians have gained popularity in the theoretical physics \cite{Lazo2017,Lazo2018,Lazo2019}, geometric mechanics \cite{deLeon2021, Rivas2022,deLeon2019,deLeon2019a,Gaset2020} and control \cite{deLeon2020, Maschke2018, vanderSchaft2018} communities, since they can be used for modelling dissipative systems of particles and fields, as well as certain thermodynamical systems.

The possible trajectories considered both in Hamilton’s and Herglotz's principles are usually smooth curves. Nevertheless, many mechanical systems of interest have non-smooth trajectories. As a matter of fact, the trajectories described by a system with impacts are not smooth, as it happens with the so-called hybrid systems.  
Mechanical systems with impacts are usually modeled as hybrid systems. Hybrid systems are dynamical systems with continuous-time and discrete-time components in their dynamics. This class of dynamical systems is capable of modeling several physical systems, such as UAVs (unmanned aerial vehicles) systems \cite{Lee2013} and bipedal robots among many others \cite{Westervelt2018,Goebel2012,Schaft2000,EyreaIrazu2021,Colombo2020,Colombo2020a,Colombo2022,Colombo2022a,Bloch2017,EyreaIrazu2022}. The disadvantage of this approach is that the impact map --the map characterizing the change of velocity in the instant of the impact-- has to be defined \textit{ad hoc} or obtained in some phenomenological fashion, e.g., by the Newtonian impact law \cite{Brogliato1996}. 
Another approach consists on considering an impulsive constraint acting on the instant of the impact \cite{Colombo2022, Lacomba1990,Ibort1997,Ibort1998,Ibort2001}. 
Alternatively, in order to characterize the dynamics of a mechanical system with impacts, one can consider a variational principle for which the possible curves are not smooth at certain points. An extension of Hamilton’s Principle to a nonsmooth setting was developed by Fetecau, Marsden, Ortiz and West \cite{Fetecau2003}. Inspired by their approach, in this paper we develop a non-smooth Herglotz principle for dissipative Lagrangian systems with impacts. The main advantage of this approach is that one can characterize the dynamics of the system with dissipation and impacts by means of just the variational principle, without the need of considering additional forces, maps or constraints.



The remainder of the paper is structured as follows. In Section \ref{sec_herglotz} Herglotz variational principle for non-smooth Lagrangian and Hamiltonian systems is presented. Section \ref{sec_billiard} presents a case study: a billard with dissipation. Section \ref{sec_contact} recalls the relation of action-dependent Lagrangians with contact Hamiltonian and Lagrangian systems. The notion of simple hybrid contact systems and their relation with action-dependent Lagrangians is explained in Section \ref{sec_hybrid}.
 We finish the paper with some outlooks in Section \ref{sec_conclusions}.



\section{Herglotz principle for nonsmooth action-dependent Lagrangians}\label{sec_herglotz}

\subsection{Herglotz variational principle for smooth Lagrangians}
Geometrically, the configuration space (the space of positions) $Q$ of a mechanical system is a differentiable manifold of dimension $n$ with local coordinates $q=(q^1,\ldots,q^n)$. At each point $q\in Q$, the tangent space to $Q$ at $q$, denoted by $T_{q}Q$, is the vector space formed by the vectors which are tangent to the curves in $Q$ passing by $q$. The union of the tangent spaces for every $q\in Q$ is called the tangent bundle of $Q$ and denoted by $TQ$. It represents the space of positions and velocities of the system. The tangent bundle has local coordinates $v_q=(q^1,\ldots, q^n,\dot{q}^{1},\ldots,\dot{q}^{n})\in TQ$ with $\hbox{dim}(TQ)=2n$ and canonical projection $\tau_Q\colon TQ \to Q$, $v_q \mapsto q$.

A Lagrangian $L:TQ\times\mathbb{R}\to\mathbb{R}$ is said to be regular if $\det W\neq 0$, where $\displaystyle{W=(W_{ij})\coloneqq \left(\frac{\partial^2L}{\partial\dot{q}^{i}\partial\dot{q}^{j}}\right)}$ for all $i,j$ with $1\leq i,j\leq n$.

Consider a Lagrangian function $L: T Q \times \RR \rightarrow \RR$ and fix two points $q_{1}, q_{2} \in Q$ and an interval $[a, b] \subset \RR$. Let us denote by $\Omega\left(q_{1}, q_{2},[a, b]\right) \subseteq\left(\mathcal{C}^{\infty}([a, b] \rightarrow Q)\right.$ the space of smooth curves $\curve$ such that $\curve(a)=q_{1}$ and $\curve(b)=q_{2}$. This space has the structure of an infinite dimensional smooth manifold whose tangent space at $\curve$ is given by the set of vector fields over $\curve$ that vanish at the endpoints \cite[Proposition 3.8.2]{Abraham2008}, that is,
\begin{equation}
\begin{aligned}
  T_{\curve} \Omega\left(q_{1}, q_{2},[a, b]\right)=&\left\{v_{\curve} \in \mathcal{C}^{\infty}([a, b] \rightarrow T Q) \mid
  \tau_{Q} \circ v_{\curve}=\curve,\right.\\&\left.\qquad v_{\curve}(a)=0,\ v_{\curve}(b)=0\right\}.
\end{aligned}
\end{equation}
The elements of $T_{\curve} \Omega\left(q_{1}, q_{2},[a, b]\right)$ will be called \emph{infinitesimal variations} of the curve $\curve$. Let
\begin{equation}
  \mathcal{Z}: \mathcal{C}^{\infty}([a, b] \rightarrow Q) \rightarrow \mathcal{C}^{\infty}([a, b] \rightarrow \RR)
\end{equation}
be the operator that assigns to each curve $\curve$ the function $\mathcal{Z}(\curve)$ that solves the following ordinary differential equation (ODE):
\begin{equation}
\begin{aligned}
    &\frac{\mathrm{d} \mathcal{Z}(\curve)(t)}{\mathrm{d} t}=L(\curve(t), \dot{\curve}(t), \mathcal{Z}(\curve)(t)), \\
    & \mathcal{Z}(\curve)(a)=0.
\end{aligned}
\end{equation}
Now we define the \emph{action functional} $\mathcal{A}$ as the map which assigns to each curve the solution to the previous ODE evaluated at the endpoint, namely,
\begin{equation}
\begin{aligned}
  \mathcal{A}: \Omega\left(q_{1}, q_{2},[a, b]\right) & \rightarrow \RR \\
\curve & \mapsto \mathcal{Z}(\curve)(b).
\end{aligned}
\end{equation}
We will say that a path $\curve \in \Omega\left(q_{1}, q_{2}, Q\right)$ satisfies the \emph{Herglotz variational principle} if it is a critical point of $\mathcal{A}$, i.e.,
\begin{equation}
  T_{\curve} \mathcal{A}=0.
\end{equation}
These critical points are curves which satisfy the \emph{Herglotz equations} \cite{Herglotz1930, Georgieva2011, deLeon2019}:
\begin{equation}
    \frac{\partial L}{\partial q^{i}}-\frac{\mathrm{d}}{\mathrm{d} t} \frac{\partial L}{\partial \dot{q}^{i}}+\frac{\partial L}{\partial \dot{q}^{i}} \frac{\partial L}{\partial z}=0.
    \label{Herglotz_eqs}
\end{equation}

Observe that Herglotz equations are like Euler-Lagrange equations with an additional term due to the dependence of the Lagrangian on the action: $\frac{\partial L}{\partial \dot{q}^{i}} \frac{\partial L}{\partial z}$. For instance, if the Lagrangian is of the form $L(q, \dot q, z) =\frac{1}{2} \dot q^T \dot q - V(q)- \gamma z $ for some $\gamma \in \RR$, this extra term is $-\gamma \dot q^i$,  corresponding to a dissipation linear in the velocities.


\subsection{Herglotz principle for nonsmooth Lagrangians} 

Since trajectories with impacts are not smooth curves, the space of curves will no longer be a smooth manifold. Therefore, Herglotz variational principle cannot be generalized to a nonsmooth setting in a straightforward manner. In order to overcome this problem, we make use of the Fetecau, Marsden, Ortiz and West's approach \cite{Fetecau2003}: extend the problem to the nonautonomous case so that both position variables and time are functions of a parameter $\tau$. In this way, the impact can be fixed in $\tau$ space while remaining variable in both configuration and time spaces. Additionally, this allows to define a path space $\mathcal M$ which is indeed a smooth manifold. By taking variations on this submanifold, we shall obtain a nonsmooth Herglotz principle.

Consider a configuration manifold $Q$ and a submanifold with boundary $C \subset Q$, which represent the subset of admissible configurations. Let $L: T Q \times \RR \rightarrow \RR$ be a regular Lagrangian. Let us introduce the \emph{path space}
\begin{equation*}
  \mathcal{M}\coloneqq \mathcal{T} \times \mathcal{Q}\left([0,1], \tau_{i}, \partial C, Q\right),
\end{equation*}
where
\begin{equation*}
\begin{aligned} 
  &\mathcal{T}\coloneqq \left\{c_{t} \in C^{\infty}([0,1], \mathbb{R}) \mid c_{t}^{\prime}>0 \text { in }[0,1]\right\},\\  
  &\mathcal{Q}\left([0,1], \tau_{i}, \partial C, Q\right)\coloneqq\left\{c_{q}:[0,1] \rightarrow Q \mid c_{q} \text { is a } C^{0},\right.\\
  &\left. \qquad\qquad\text { piecewise } C^{2} \text { curve, } 
  \right.\\
  &\qquad\left.c_{q}(\tau) \text { has only one singularity at } \tau_{i}, c_{q}\left(\tau_{i}\right) \in \partial C\right\}.
\end{aligned}
\end{equation*}
Here $\displaystyle{c_t^\prime(0) = \lim_{t\to 0^+} c_t^\prime(t)}$ and $\displaystyle{c_t^\prime(1) = \lim_{t\to 1^-} c_t^\prime(t)}$ are understood, and similarly for $c_q^\prime$ and higher derivatives.
A path $c \in \mathcal{M}$ is a pair $c=\left(c_{t}, c_{q}\right)$. Given a path, the \emph{associated curve} $q:\left[c_{t}(0), c_{t}(1)\right] \rightarrow Q$ is given by $q(t) = c_q \circ c_t^{-1} (t)$\footnote{Notice that, since $c_t^\prime>0$, each $c_t\in \mathcal{T}$ is injective and thus invertible (if it is not surjective, it suffices to restrict the codomain), so $q$ is well-defined. }.
Let $\mathcal C$ denote the set of all paths $q(t) \in Q$.

The moment of impact $\tau_i\in (0,1)$ is fixed in the $\tau$-space, but can vary in the $t$-space according to $t_i=c_t(\tau_i)$. One can show that $\mathcal{T}$ and $\mathcal{Q}\left([0,1], \tau_{i}, \partial C, Q\right)$, and hence $\mathcal{M}$, are smooth manifolds \cite{Fetecau2003}.  The tangent space at $c_q\in \mathcal Q$ is given by
\begin{align*}
  T_{c_q} \mathcal{Q}=&\left\{v:[0,1] \rightarrow T Q \mid v \text { is a } C^{0} \text { piecewise } C^{2} \operatorname{map},\right.\\ & \left.\qquad\qquad\qquad v\left(\tau_{i}\right) \in T_{c_q\left(\tau_{i}\right)} \partial C\right\}.
\end{align*}

Let $\hat \Omega \left(q_1, q_2, [0,1]  \right)\subset \mathcal{M}$ be the subset of curves such that $c_q(0)=q_1$ and $c_q(1)=q_2$.
Consider the operator
\begin{equation*}
  \hat{\mathcal{Z}}: \hat \Omega \left(q_1, q_2, [0,1]  \right) \to \mathcal{T}
\end{equation*}
that assigns to each $c_q\in \mathcal{M}$ the solution of the following ODE:
\begin{equation*}
\begin{aligned} 
  & \frac{\mathrm{d} \hat{\mathcal {Z}}} {\mathrm{d} \tau} 
  = L \left( c_q(\tau), \frac{c_q'(\tau)}{c_q'(\tau)}, \hat{\mathcal {Z}}(c_q, c_t) (\tau)   \right) c_t'(\tau)  ,\\
  &\hat{\mathcal {Z}}(c_q, c_t)(0) = \hat{z}_0,
\end{aligned}
\end{equation*}
and denote by $\hat{\mathcal {A}}$ the functional
\begin{equation*}
\begin{aligned} 
  \hat{\mathcal {A}}:   \hat \Omega \left(q_1, q_2, [0,1]  \right) &\to \RR\\
  (c_q,c_t) & \mapsto \hat{\mathcal {Z}}(c_q, c_t) (1).
\end{aligned}
\end{equation*}

\begin{theorem}[Nonsmooth Herglotz variational principle] \label{theorem_nonsmooth_lagrangian}
Let $L:TQ \times \RR$ be a smooth and regular Lagrangian function. 
Let $c=(c_q, c_t)$ be a curve in $\hat \Omega \left(q_1, q_2, [0,1]  \right)$, and let $\chi (\tau)= \left(c_q(\tau), \frac{c_q'(\tau)}{c_t'(\tau)}, \hat{\mathcal{Z}}(c)(\tau)\right) \subset TQ\times \RR$. Then, $c$ is a critical point of $\hat{\mathcal A}$ if and only if
\begin{subequations}
\begin{flalign}
  &\frac{\partial L} {\partial q^i} (\chi(\tau)) 
    -\frac{\mathrm{d} } {\mathrm{d}t} \frac{\partial L} {\partial \dot q^i } (\chi(\tau))  
    \nonumber\\
    &\qquad \qquad \ +\frac{\partial L} {\partial \dot q^i } (\chi(\tau))  \frac{\partial L} {\partial  z} (\chi(\tau))   = 0, \label{equation_Herglotz} \\
    &\frac{\mathrm{d} } {\mathrm{d}t} E_L(\chi(\tau))
    = \frac{\partial L} {\partial z} (\chi(\tau))\ E_L(\chi(\tau)),
    \label{equation_energy_dissipation}
\end{flalign}
\end{subequations}
for $\tau \in [0, \tau_i) \cup (\tau_i,1]$, and
\begin{equation}
\begin{aligned}
  & \frac{\partial L} {\partial \dot q^i} (\chi(\tau_i^-))  v^i
  = \frac{\partial L} {\partial \dot q^i} (\chi(\tau_i^+)) v^i, \\
  & E_L (\chi(\tau_i^-))   = E_L (\chi(\tau_i^+)) ,
\end{aligned} \label{conditions_impact}
\end{equation}
where $\displaystyle{\chi(\tau_i^\pm) = \lim_{\tau \to \tau_i^\pm} \chi(\tau)}$, for any $v \in T_{c_q(\tau_i)}\partial C$. 
\end{theorem}

\begin{proof}
Let $c=(c_q,c_t)\in  \hat \Omega \left(q_1, q_2, [0,1]  \right)$ be a curve. 
Consider a smoothly parametrized family of curves $c^\lambda=(c_q^\lambda, c_t^\lambda)$ in $\hat \Omega \left(q_1, q_2, [0,1]  \right)$ such that $c^0=c$,
  $
  u = \left.\frac{\mathrm{d}c_q^\lambda} {\mathrm{d}\lambda} \right|_{\lambda=0}$, and
  $\theta = \left.\frac{\mathrm{d}c_t^\lambda} {\mathrm{d}\lambda} \right|_{\lambda=0}.
  $
Let $\varphi=T_{c} \hat{\mathcal{Z}}(u,\theta)$, so that $T_c \hat{\mathcal{A}}(u, \theta)=\varphi(1)$. Observe that $\varphi(0)=0$, since $\hat{\mathcal{Z}}(c^\lambda)(0)= \hat{z}_0$ for every $\lambda$. 
We have that
\begin{equation*}
\begin{aligned} 
  \varphi^{\!\prime}(\tau)
  & = \left. \frac{\mathrm{d} } {\mathrm{d}\tau} \frac{\mathrm{d} } {\mathrm{d}\lambda}   \hat{\mathcal{Z}}(c^\lambda(\tau))\right|_{\lambda=0}
  = \left.  \frac{\mathrm{d} } {\mathrm{d}\lambda} \frac{\mathrm{d} } {\mathrm{d}\tau}  \hat{\mathcal{Z}}(c^\lambda(\tau))\right|_{\lambda=0}\\
  &= \left.  \frac{\mathrm{d} } {\mathrm{d}\lambda}  \left[
    L \left( c_q^\lambda(\tau), \frac{c_q^{\lambda \prime}(\tau)}{c_q^{\lambda \prime}(\tau)}, \hat{\mathcal {Z}}(c_q^\lambda, c_t^\lambda) (\tau)   \right) c_t^{\lambda \prime}(\tau) 
    \right] \right|_{\lambda=0}\\
  & = \left[ \frac{\partial L} {\partial q^i} (\chi(\tau)) u^i(\tau) 
      + \frac{\partial L} {\partial \dot q^i } (\chi(\tau)) \frac{1}{c_t'(\tau)} u^{\prime i} (\tau)  \right.\\ &\left. \quad
      -\frac{\partial L} {\partial \dot q^i } (\chi(\tau)) \frac{c_q'(\tau)}{c_t'(\tau)^2} \theta(\tau)
      + \frac{\partial L} {\partial z} (\chi(\tau)) \varphi(\tau)
   \right] c_t'(\tau)\\
&  +  L \left( c_q^\lambda(\tau), \frac{c_q^{\lambda \prime}(\tau)}{c_q^{\lambda \prime}(\tau)}, \hat{\mathcal {Z}}(c_q^\lambda, c_t^\lambda) (\tau)   \right) \theta'(\tau).
\end{aligned}
\end{equation*}

An integrating factor for this ODE is 
$
  \mu (\tau) = \exp \left( - \int_0^\tau \frac{\partial L} {\partial z} (\chi(s))\ c_t'(s) \dd s  \right),
  $
so
\begin{equation*}
\begin{aligned} 
    \varphi(\tau) \mu(\tau) 
    &= \int_0^\tau \mu(s) c_t'(s)
    \left[ \frac{\partial L} {\partial q^i} (\chi(s)) u^i(s) 
      \right.\\&\left.+ \frac{\partial L} {\partial \dot q^i } (\chi(s)) \frac{1}{c_t'(s)} u^{\prime i} (s)-\frac{\partial L} {\partial \dot q^i } (\chi(s)) \frac{c_q'(s)}{c_t'(s)^2} \theta'(s) \right] \ \dd s \\&+ \int_0^\tau \mu(s) L (\chi(s)) \theta'(s)\ \dd s\\
    &= \int_0^\tau \mu(s) c_t'(s)  u^i(s)
    \left[ \frac{\partial L} {\partial q^i} (\chi(s))\right.\\ &\left.
+ \frac{\partial L} {\partial \dot q^i } (\chi(s)) \frac{1}{c_t'(s)}
       \right] \ \dd s   \\  
  & \quad + \int_0^\tau \mu(s) \theta'(s)
  \left[ L (\chi(s)) 
  -\frac{\partial L} {\partial \dot q^i } (\chi(s)) \frac{c_q'(s)}{c_t'(s)}
    \right]\dd s.
\end{aligned}
\end{equation*}
Integrating by parts and taking into account that $u(0)=u(1)=0$ and $\theta(0)=\theta(1)=0$, we obtain
\allowdisplaybreaks
\begin{align} 
  \varphi(1)& \mu (1)
  = \int_0^{\tau_i} \mu(s) c_t'(s)  u^i(s)
    \left[ \frac{\partial L} {\partial q^i} (\chi(s)) 
    -\frac{\mathrm{d} } {\mathrm{d}t} \frac{\partial L} {\partial \dot q^i } (\chi(s))  \right.\\
    &\qquad\left.+\frac{\partial L} {\partial \dot q^i } (\chi(s))  \frac{\partial L} {\partial  z} (\chi(s)) \right] \ \dd s   \\  
  & \quad - \int_0^{\tau_i} \mu(s) \theta(s) 
  \frac{\mathrm{d} } {\mathrm{d}s} \left[ L (\chi(s)) 
  -\frac{\partial L} {\partial \dot q^i } (\chi(s)) \frac{c_q'(s)}{c_t'(s)}
    \right] \dd s\\
    & \quad + \int_0^{\tau_i} \mu(s) \theta(s) 
  \left[ L (\chi(s)) 
  \right.\\ &\left.\qquad
  -\frac{\partial L} {\partial \dot q^i } (\chi(s)) \frac{c_q'(s)}{c_t'(s)}
    \right] \frac{\partial L} {\partial z} (\chi(s))\ c_t'(s)  \dd s\\
  & \quad + \int_{\tau_i}^1 (\cdots)\ \dd s\\
  &\quad 
  +\left. \frac{\partial L} {\partial \dot q^i} (\chi(s)) u^i(s) \mu(s) \right|_{\tau_i^+}^{\tau_i^-}
  \\&+ \left.\mu(s) \theta(s)  \left[ L (\chi(s)) 
  -\frac{\partial L} {\partial \dot q^i } (\chi(s)) \frac{c_q'(s)}{c_t'(s)}
    \right] \right|_{\tau_i^+}^{\tau_i^-}.
\end{align}
Since $\mu(\tau)$ is nonzero, $\varphi(0)$ vanishes for every $(u,\theta)$ (i.e., $\chi$ is a critical point of $\hat{\mathcal{A}}$) if and only if
\begin{subequations}
\begin{flalign*}
  &\frac{\partial L} {\partial q^i} (\chi(\tau)) 
    -\frac{\mathrm{d} } {\mathrm{d}t} \frac{\partial L} {\partial \dot q^i } (\chi(\tau))  
    +\frac{\partial L} {\partial \dot q^i } (\chi(\tau))  \frac{\partial L} {\partial  z} (\chi(\tau))   = 0,\\
    &\frac{\mathrm{d} } {\mathrm{d}\tau} \left[ L (\chi(\tau)) 
  -\frac{\partial L} {\partial \dot q^i } (\chi(\tau)) \frac{c_q'(\tau)}{c_t'(\tau)}
    \right]\\
\qquad &= \left[ L (\chi(\tau)) 
  -\frac{\partial L} {\partial \dot q^i } (\chi(\tau)) \frac{c_q'(\tau)}{c_t'(\tau)}
    \right] \frac{\partial L} {\partial z} (\chi(\tau))\ c_t'(\tau),
\end{flalign*}
\end{subequations} for $\tau \in [0, \tau_i) \cup(\tau_i,1]$, and
\begin{subequations}
\begin{flalign*}
   \frac{\partial L} {\partial \dot q^i} (\chi(\tau_i^-))  v^i
  &= \frac{\partial L} {\partial \dot q^i} (\chi(\tau_i^+)) v^i, \\
   L (\chi(\tau_i^-)) 
  &-\frac{\partial L} {\partial \dot q^i } (\chi(\tau_i^-)) \frac{c_q'(\tau_i^-)}{c_t'(\tau_i^-)}
 \\& = L (\chi(\tau_i^+)) 
  -\frac{\partial L} {\partial \dot q^i } (\chi(\tau_i^+)) \frac{c_q'(\tau_i^+)}{c_t'(\tau_i^+)},
\end{flalign*}
\end{subequations}
for any $v \in T_{c_q(\tau_i)}\partial C$. The result follows from the chain rule and the definition of $E_L$.
\end{proof}

\begin{remark}
Equation \eqref{equation_energy_dissipation} is redundant. Indeed, we have that
\begin{equation*}
\begin{aligned} 
  \frac{\mathrm{d} E_L  } {\mathrm{d}t} 
  &= \frac{\mathrm{d} } {\mathrm{d}t}
  \left( \frac{\partial L} {\partial \dot q^i}  \dot q^i -L  \right) \\ 
 & =  \left( \frac{\mathrm{d} } {\mathrm{d}t}\frac{\partial L} {\partial \dot q^i}   - \frac{\partial L} {\partial q^i}  \right) \dot q^i
  - \frac{\partial L} {\partial z} \dot z\\
&  = \left( \frac{\partial L} {\partial \dot q^i}  \dot q^i -L  \right) \frac{\partial L} {\partial z}
  = E_L \frac{\partial L} {\partial z},
\end{aligned}
\end{equation*}
along solutions of the Herglotz equations \eqref{equation_Herglotz}. This equation expressed the rate of dissipation of the energy.
\end{remark}

As it is well-known, the conserved quantities of an action-independent Lagrangian can be used to integrate its equations of motion. Although action-dependent Lagrangians may have conserved quantities as well, it is more natural to consider the so-called dissipated quantities. They can be used to integrate the equations of motion of an action-dependent Lagrangian system (see Section \ref{sec_billiard}).
\begin{definition}
Given an action-dependent Lagrangian function $L\colon TQ \times  R \to R$, a \emph{dissipated quantity} is a function $f\colon TQ \times  R \to R$ which is dissipated at the same rate as the energy, namely,
\begin{equation}
  \frac{\mathrm{d}} {\mathrm{d}t} f\left(q(t), \dot q(t), z(t)\right) =  \frac{\partial L} {\partial z}\left(q(t), \dot q(t), z(t)\right)\, f \left(q(t), \dot q(t), z(t)\right),
\end{equation}
where $\left(q(t), \dot q(t), z(t)\right)$ is a curve on $TQ \times \R$ that satisfies Herglotz equations.
\end{definition}


\subsection{Nonsmooth Hamiltonian equations for contact Hamiltonian systems}
Given a Lagrangian function $L: T Q \times \mathbb{R} \rightarrow \mathbb{R}$, we can define the \emph{Legendre transform} $\Leg: T Q \times \RR \to T^{*} Q \times \RR$ by
\begin{equation}
\begin{aligned} 
  \Leg: 
  \left( q^i, \dot q^i, z  \right) & \mapsto \left( q^i, \frac{\partial L} {\partial \dot q^i}, z  \right).
\end{aligned}
\end{equation}
Hereinafter, we will assume that the Lagrangian $L$ is \emph{hyper-regular}, i.e., the Legendre transform is a diffeomorphism. The Hamiltonian function $H$ is then given by $H=E_L\circ \Leg^{-1}$.

The Hamiltonian counterpart of Theorem \ref{theorem_nonsmooth_lagrangian} is as follows.

\begin{proposition}[Hamiltonian nonsmooth Herglotz principle]\label{Herglotz_Hamiltonian}
Let $H$ be a regular Hamiltonian function on $T^*Q\times \RR$.
Let $\xi=(q^i, p_i, \mathcal Z)$ be a continuous and piecewise $C^2$ curve on $T^*Q\times \RR$ whose only singularities occur at $t_i$. Then, 
\begin{equation}
\begin{array}{ll} 
   \dfrac{\mathrm{d} q^i} {\mathrm{d}t} = \dfrac{\partial H} {\partial p_i}(\xi(t)),
  & \dfrac{\mathrm{d}p_i} {\mathrm{d}t}
  = - \dfrac{\partial H} {\partial q^i}(\xi(t)) - p_i \dfrac{\partial H} {\partial z}(\xi(t)),
\end{array}
\label{eqs_vector_Hamiltonian_Herglotz}
\end{equation}
for $t\neq t_i$, and
\begin{equation}
\begin{array}{ll} 
  p_i (\xi (t_i^-))  v^i 
  = p_i (\xi (t_i^+))  v^i ,
  &H (\xi (t_i^-)) = H  (\xi (t_i^+)),
\end{array}
\label{eqs_impact_Hamiltonian_Herglotz}
\end{equation}
where $\displaystyle{\xi(\tau_i^\pm) = \lim_{\tau \to \tau_i^\pm} \xi(\tau)}$.

\end{proposition}


\section{Case Study: Billard with dissipation}\label{sec_billiard}

Consider a particle moving in the plane confined to the surface $C\subset \RR^2$ defined by $x^2+y^2\leq1$. The Lagrangian $L:T\RR^2 \times \RR\to \RR$ is given by
\begin{equation}
  L (x, y, \dot x, \dot y, z)
  = \frac{1}{2} \left( \dot x^2 + \dot y^2  \right)
  - \gamma z,
\end{equation}
where $\gamma$ is a real constant.
Herglotz equations thus yield
\begin{equation}
\begin{array}{lll} 
  \ddot x = - \gamma  \dot x,
  &\ddot y = - \gamma  \dot y,
  & \dot z = L (x, y, \dot x, \dot y, z).
\end{array}
\end{equation}
Their solutions for the initial conditions $x(0)=x_0,\ y(0)=y_0,\ \dot x(0)=\dot x_0,\ \dot y(0)=\dot y_0$ and $E_L(0)=E_0$ are
\begin{equation*}
\begin{aligned} 
    &x(t) = x_0 + \frac{\dot x_0}{\gamma} \left( 1- e^{-t\gamma}  \right),\\
    &y(t) = y_0 + \frac{\dot y_0}{\gamma} \left( 1- e^{-t\gamma}  \right),\\
    &z(t) = - \frac{1}{\gamma} \left[ \frac{1}{2} \left( \dot x^2_0 + \dot y^2_0  \right) e^{-2t\gamma} +E_0 e^{-t\gamma} \right].
\end{aligned}
\end{equation*}

Observe that $v\in T_{(x,y)} \partial C$ if and only if $\dd h(v)=0$, where $h(x,y)=1-x^2-y^2$ is the function characterizing the surface $C$, so $v=v_x \partial_x -\tfrac{x}{y}v_x \partial_y$.
Conditions \eqref{conditions_impact} can then be written as
\begin{equation*}
\begin{aligned} 
  & \dot x^- -\frac{x}{y} \dot y^- = \dot x^+ -\frac{x}{y} \dot y^+ , \\
  & \frac{1}{2}\left( (\dot x^-)^2 + (\dot y^-)^2 \right) + \gamma z_{\mid t_i^-}
  =  \frac{1}{2}\left( (\dot x^+)^2 + (\dot y^+)^2 \right) + \gamma z_{\mid t_i^+},
\end{aligned}
\end{equation*}
where $\dot x^\pm = \dot x(t_i^\pm)$ and $\dot y^\pm = \dot y(t_i^\pm)$. Now, $z_{\mid t_i^+}=z_{\mid t_i^-}$, so
\begin{align}
    & \dot x^+ = \frac{-\dot x^- x^2+\dot x^- y^2-2 \dot y^- x y}{x^2+y^2},\\
    & \dot y^+ = \frac{-2 \dot x^- x y+\dot y^- x^2-\dot y^- y^2}{x^2+y^2}.
\end{align}
In polar coordinates,
\begin{equation}
    L(r, \theta, \dot r, \dot \theta,z ) = \frac{1}{2}\left(\dot r^2 + r^2 \dot \theta^2 \right) - \gamma z.
\end{equation}
One can check that $\ell = r^2\dot \theta$ is a dissipated quantity, namely,
\begin{equation}
    \frac{\dd \ell}{\dd t} = \frac{\partial L}{\partial z} \ell = -\gamma \ell,
\end{equation}
so we can write $\ell = \ell_0 e^{-\gamma t}$, where $\ell_0=\ell(t=0)$. The equations of motions outside the impact surface can thus be expressed as
\begin{equation}
\begin{aligned}
    &\ddot r = -\gamma \dot r + \frac{\ell_0}{r}e^{-\gamma t},\\
    &\dot \theta = \frac{\ell_0}{r^2}e^{-\gamma t}.   
\end{aligned}
\label{eqs_motion_polar}
\end{equation}
On the other hand, deriving $r^2=x^2+y^2$ we obtain
$2r\dot r=2(x\dot x+ y\dot y)$, so
\begin{equation}
\begin{aligned}
    r\dot r^+ &= x\dot x^+ + y\dot y^+\\
    &= \frac{x\left(-\dot x^- x^2+\dot x^- y^2-2 \dot y^- x y\right)}{x^2+y^2} \\
    & + \frac{y\left(-2 \dot x^- x y+\dot y^- x^2-\dot y^- y^2\right)}{x^2+y^2} \\
    & = -\left( x\dot x^- +y \dot y^-\right) = -r\dot r^-,
\end{aligned}
 \label{eqs_impact_r}
\end{equation}
hence $\dot r^+ = -\dot r^-$. Similarly, deriving $\theta=\arctan(y/x)$ yields
\begin{equation}
\begin{aligned}
    \dot \theta^+ &= \frac{1}{1+(y/x)^2} \left(\frac{x\dot y^+  - y \dot x^+}{x^2}\right)\\
    &= \frac{x\left(-2 \dot x^- x y+\dot y^- x^2-\dot y^- y^2\right)}{(x^2+y^2)^2} \\
    &+ \frac{y\left(-\dot x^- x^2+\dot x^- y^2-2 \dot y^- x y\right)}{(x^2+y^2)^2} \\
    &= \frac{(-y\dot x^- + x \dot y^-)(x^2+y^2)}{(x^2+y^2)^2} = \dot \theta^-.
\end{aligned}
 \label{eqs_impact_theta}
\end{equation}
Making use of Eqs.~\eqref{eqs_motion_polar}, \eqref{eqs_impact_r} and \eqref{eqs_impact_theta} we have performed a python numerical simulation for $\gamma = 10^{-4}$ and initial values $x(0)=0.5,\ y(0)=0,\ \dot x(0)=1,\ \dot y(0)=1$.
\begin{figure}[t]
    \centering
    \includegraphics[width=.6\linewidth]{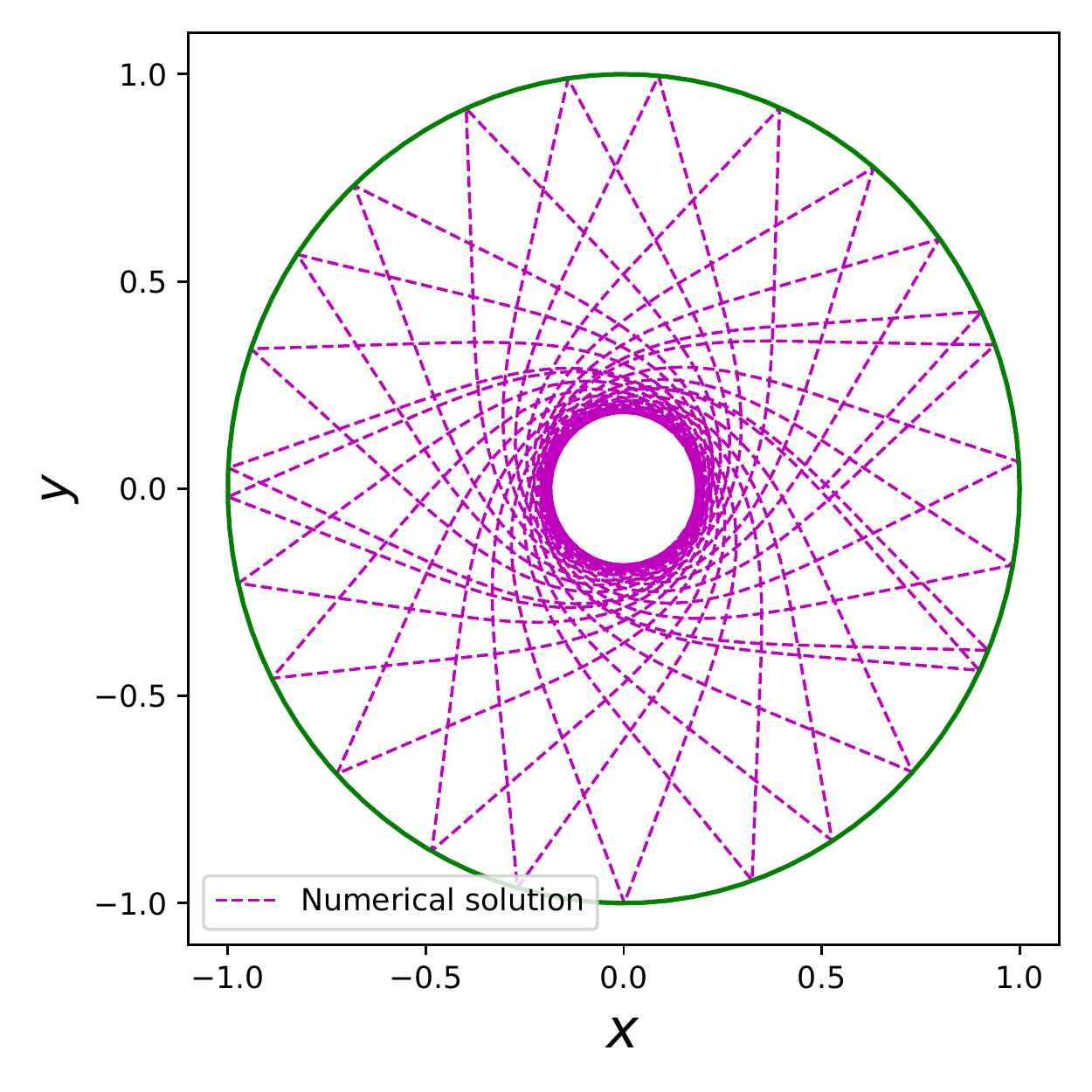}
    \caption{Numerical simulation for the trajectory of a particle in the billiard, with $\gamma = 10^{-4}$.}
    \label{fig:simulation}
\end{figure}


Similarly, the Hamiltonian
$H:T^\ast \RR^2 \times \RR\to \RR$ is given by
\begin{equation}
  H (x, y, p_x, p_y, z)
  = \frac{1}{2} \left( p_x^2 + p_y^2  \right)
  + \gamma z.
\end{equation}
From Proposition \ref{Herglotz_Hamiltonian} we can obtain the trajectories of the system:
\begin{equation}
\begin{array}{llll}
   \dfrac{\mathrm{d} x} {\mathrm{d}t} 
 = p_x,
&\dfrac{\mathrm{d} y} {\mathrm{d}t} 
 = p_y,
 &
  \dfrac{\mathrm{d}p_x} {\mathrm{d}t}
= -p_x \gamma,
  & \dfrac{\mathrm{d}p_y} {\mathrm{d}t}
= -p_x \gamma,
\end{array} 
\end{equation}
for $t\neq t_i$, and
\begin{align}
    & p_x^+ = \frac{-p_x^- x^2+p_x^- y^2-2 p_y^- x y}{x^2+y^2},\\
    & p_y^+ = \frac{-2 p_x^- x y+p_y^- x^2-p_y^- y^2}{x^2+y^2},
    \label{impact_map_billiard}
\end{align}
where $p_x^\pm = p_x(t_i^\pm)$ and $p_y^\pm = p_y(t_i^\pm)$. 

\subsection{Elliptical billiard}

Suppose now that the particle is confined to the surface $C=\left\{\left(\frac{x}{a}\right)^2 + \left(\frac{y}{b}\right)^2\leq 1\right\}$, where $a$ and $b$ are positive constants. Then, from conditions \eqref{conditions_impact} we obtain
\begin{align}
   &\dot x^+= \frac{a^4 \dot x^-  y^2-2 a^2 b^2 \dot y^-  x y-b^4 \dot x^-  x^2}{a^4 y^2+b^4 x^2},\\
   & \dot y ^+ = \frac{-a^4 \dot y^-  y^2-2 a^2 b^2 \dot x^-  x y+b^4 \dot y^-  x^2}{a^4 y^2+b^4 x^2},
\end{align}
or, in polar coordinates, 
\begin{equation}
\begin{aligned}
  \dot r^+ =&\ \frac{r}{4 r  \left(a^4 \sin ^2(\theta )+b^4 \cos ^2(\theta )\right)}
            \\&\times
      \left[2 r \dot \theta^-  \left(b^4-a^4\right) \sin (2 \theta )+2 \dot r^-  \left(a^4-b^4\right) \cos (2 \theta )
            \right. \\ & \left.
      +r \dot \theta^-  \left(a^2-b^2\right)^2 \sin (4 \theta )-\dot r^-  \left(a^2-b^2\right)^2 \cos (4 \theta )
            \right. \\ & \left.
      -\dot r^-  \left(a^2+b^2\right)^2\right],
      \label{eqs_impact_r_ellipse}
\end{aligned}
\end{equation}
and
\begin{equation}
\begin{aligned}
      \dot \theta^+ =&\ \frac{r}{4 r  \left(a^4 \sin ^2(\theta )+b^4 \cos ^2(\theta )\right)}
            \\ &\times
     \left[2 r \dot \theta^-  \left(b^4-a^4\right) \sin (2 \theta )+2 \dot r^-  \left(a^4-b^4\right) \cos (2 \theta )
             \right. \\ & \left.
     +r \dot \theta^-  \left(a^2-b^2\right)^2 \sin (4 \theta )-\dot r^-  \left(a^2-b^2\right)^2 \cos (4 \theta )
             \right. \\ & \left.
     -\dot r^-  \left(a^2+b^2\right)^2\right].
     \label{eqs_impact_theta_ellipse}
\end{aligned}
\end{equation}
Making use of Eqs.~\eqref{eqs_motion_polar}, \eqref{eqs_impact_r_ellipse} and \eqref{eqs_impact_theta_ellipse} we have performed a python numerical simulation for $\gamma = 10^{-4},\ a=0.9, b=1.1$, and initial values $x(0)=0.5,\ y(0)=0,\ \dot x(0)=1,\ \dot y(0)=1.2$.
\begin{figure}[t]
    \centering
    \includegraphics[width=.6\linewidth]{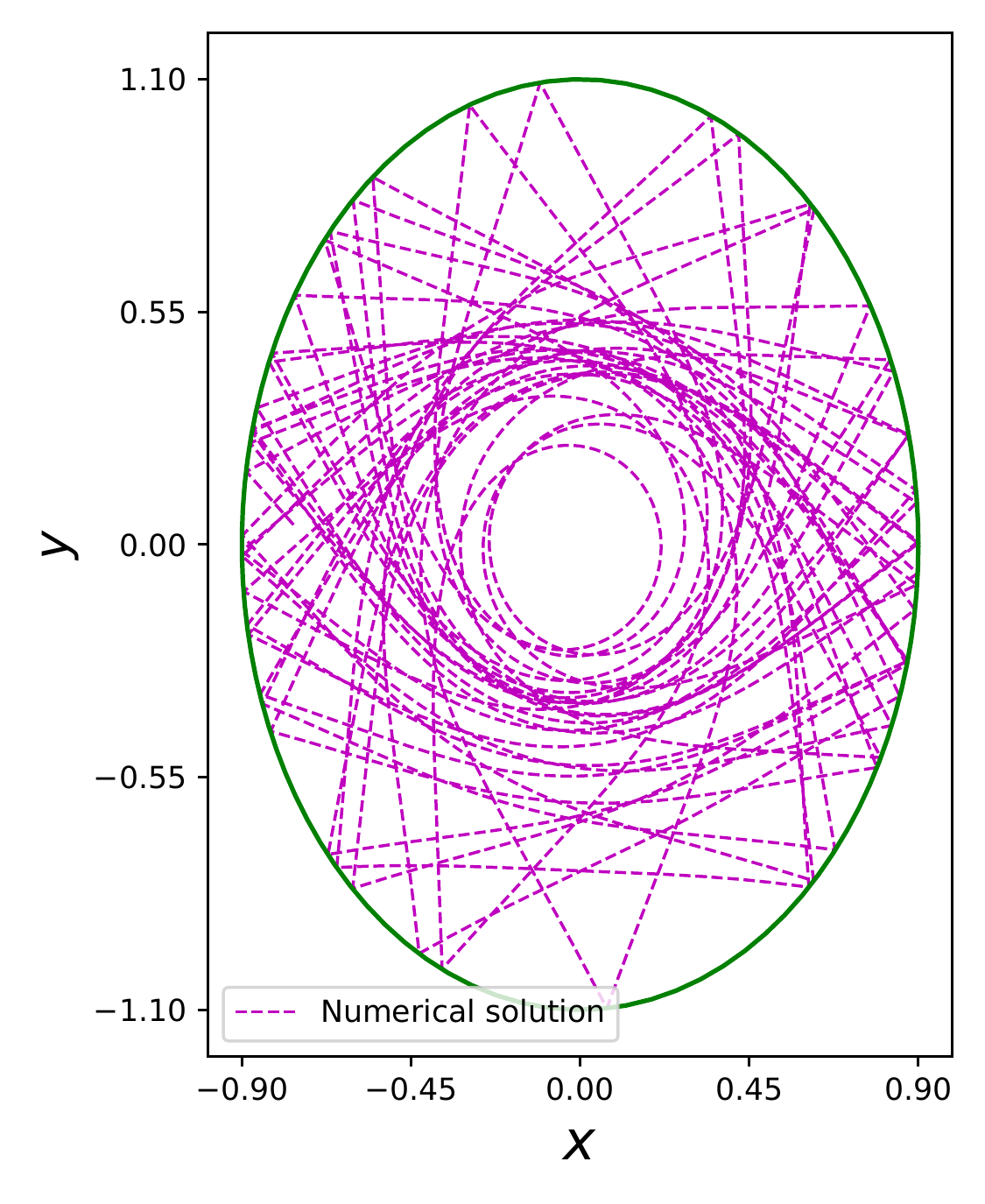}
    \caption{Numerical simulation for the trajectory of a particle in an elliptical billiard, with $\gamma = 10^{-4}$.}
    \label{fig:simulation}
\end{figure}

\section{Relation with contact Lagrangian and Hamiltonian systems}\label{sec_contact}
In this section, we briefly recall the relation between action-dependent Lagrangians and contact Hamiltonian and Lagrangian systems, in order to relate Nonsmooth Herglotz principle with hybrid contact systems in the following section. In particular, we recall that the dynamics of a regular contact Lagrangian system are given by the Herglotz equations. For more details, we refer to \cite{deLeon2021a, Rivas2022,deLeon2019,Gaset2020,deLeon2019a,deLeon2021} and references therein.

Symplectic geometry is the natural framework for Hamiltonian mechanics, and regular Lagrangians define a symplectic structure on $TQ$ as well. Similarly, contact geometry is the natural framework
for Hamiltonian systems with dissipation (the so-called contact Hamiltonian systems) and regular action-dependent Lagrangians.

A \emph{contact manifold} is a pair $(M, \eta)$, where $M$ is an $(2 n+1)$-dimensional manifold and $\eta$ is a $1$-form on $M$ such that $\eta \wedge(\mathrm{d} \eta)^{n}$ is a volume form. This $1$-form $\eta$ is called a \emph{contact form}. Given a contact manifold $(M,\eta)$, there exists a unique vector field $\Reeb$ on $M$ such that
\begin{equation*}
   \contr{\Reeb} \dd \eta = 0,\,\,\,  
   \contr{\Reeb} \eta = 1.
\end{equation*}
The vector field $\Reeb$ is called the \emph{Reeb vector field}.

In \emph{Darboux coordinates} $(q^i, p_i, z)$, the contact form is locally written
\begin{equation*}
  \eta = \dd z - p_i \dd q^i,
\end{equation*}
and the Reeb vector field is $\Reeb = \frac{\partial  } {\partial z}$. The contact structure $\eta$ on $M$ defines the \emph{musical isomorphisms} $\flat : TM \to T^*M$
\begin{equation*}
\begin{aligned}
  \flat \colon
  v & \mapsto \contr{v} \dd \eta + \eta(v) \eta,
\end{aligned}
\end{equation*}
and $\sharp = \flat^{-1}$.

Given a Hamiltonian function $H$ on $(M,\eta)$, we define the \emph{(contact) Hamiltonian vector field} $X_H$ by
\begin{equation}
  \flat(X_H) = \dd H - \left(\Reeb (H) + H  \right) \eta.
  \label{Hamiltonian_vector_field}
\end{equation}
The triple $(M,\eta, H)$ is called a \emph{(contact) Hamiltonian system}. In Darboux coordinates,
\begin{equation}
  X_H =   \frac{\partial H}{ \partial p_i} \frac{\partial  } {\partial q^i} - \left(\frac{\partial H}{\partial q^i} + p_i\frac{\partial H}{\partial z}\right)\frac{\partial}{\partial p_i} + \left(p_i\frac{\partial H}{\partial p_i} - H\right)\frac{\partial }{\partial z}.
\end{equation}
Additionally, the following identities hold:
\begin{subequations}
\begin{flalign}
  &X_H(H) = -\Reeb(H)\ H \label{lieder_Hamiltonian},\quad \contr{X_H} \dd \eta = \dd H- \Reeb(H) \eta.
\end{flalign} \label{identities_Hamiltonian_vec}
\end{subequations}

Let $Q$ be an $n$-dimensional manifold. Consider a Lagrangian function $L:TQ\times \RR \to \RR$ and let us introduce the 1-form
$\alpha_L=S^*(\dd L)$, where $S^*$ is the adjoint operator of the vertical endomorphism $S$ on $TQ$ extended in the natural way to $TQ\times \RR$. Let $(q^i)$ be local coordinates on $Q$ and let $(q^i, \dot q^i, z)$ be the induced coordinates on $TQ\times \RR$.
 Locally, $S = \dd q^i \otimes \frac{\partial  } {\partial  \dot{q}^i}$, so $\alpha_L = \frac{\partial L} {\partial \dot q^i} \dd q^i$. Let $\eta_L$ be a 1-form on $TQ\times \RR$ given by
$\displaystyle{\eta_L = \dd z - \alpha_L = \dd z - \frac{\partial L}{\partial \dot q^i} \dd q^i}$. 

One can show that $\eta_L$ is a contact form if and only if $L$ is \emph{regular}, i.e., the Hessian matrix $(W_{ij}) = \left(\frac{\partial ^2 L} {\partial \dot q^i \dot q^j}  \right)$ is regular. Hereinafter, we shall assume that $L$ is regular. The \emph{energy} of the system is given by $E_L = \Delta(L) - L$, 
where $\Delta = \dot q^i \partial / \partial \dot q^i$ is the Liouville vector field on $TQ$ naturally extended to $TQ\times \RR$. Hence, $(TQ\times \RR, \eta_L, E_L)$ is a contact Hamiltonian system. 
Its corresponding Reeb vector field $\Reeb_L$ is locally 
\begin{equation*}
  \Reeb_L = \frac{\partial  } {\partial z} - W^{ij} \frac{\partial ^2 L } {\partial \dot q^i \partial z} \frac{\partial  } {\partial \dot q^j}, 
\end{equation*}
where $(W^{ij})$ is the inverse matrix of $(W_{ij})$. The dynamics is given by the \emph{Lagrangian vector field} $\sode_L$, defined by
\begin{equation}
  \flat_L (\sode_L) = \dd E_L - \left(E_L + \Reeb_L(E_L)  \right) \eta_L,
  \label{SODE_energy}
\end{equation}
where $\flat_L$ denotes the musical isomorphism defined by the contact form $\eta_L$. 
Eqs.~\eqref{identities_Hamiltonian_vec} are now written as
\begin{flalign}
  &\Gamma_L(E_L) = -\Reeb_L(E_L)\ E_L \label{lieder_energy},\\
  &\contr{\Gamma_L} \dd \eta_L = \dd E_L- \Reeb_L(E_L) \eta.
\end{flalign} \label{identities_Lagrangian_SODE}

\begin{remark}
Equations \eqref{lieder_energy} and \eqref{equation_energy_dissipation} are equivalent. As one can check straightforwardly, 
$
  \Reeb_L (E_L) = - \frac{\partial L} {\partial z},
$
so
\begin{equation}
  \frac{\mathrm{d}E_L} {\mathrm{d}t} = - \Reeb_L(E_L) E_L
\end{equation}
along the solutions of the Herglotz equations \eqref{equation_Herglotz}.
\end{remark}

Let us recall that a vector field $\sode$ on $TQ\times \RR$ is called a \emph{SODE} (an acronym for \emph{second order differential equation}) if $S(\sode) = \Delta$. 

If the Lagrangian $L$ is regular, one has the following equivalence between the Herglotz equations and the geometric dynamical equations (see Ref.~\cite{deLeon2019}).

\begin{theorem}
Let $L$ be a regular contact Lagrangian system on $TQ\times \RR$, and let $\sode_L$ be the Hamiltonian vector field associated with the energy (given by Eq.~\eqref{SODE_energy}). Then
\begin{enumerate}
\item $\sode_L$ is a SODE on $TQ\times \RR$,
\item the integral curves of $\sode_{L}$ are solutions of the Herglotz equations \eqref{Herglotz_eqs}.
\end{enumerate}
\end{theorem}

If $L$ is regular, locally
\begin{align}
     \sode_L &=  \dot{q}^i\parder{}{q^i} + L\parder{}{z}\\
     &+ W^{ji}\left( \parder{L}{q^j}- \dot{q}^k\parderr{L}{q^k}{\dot{q}^j} - L\parderr{L}{z}{\dot{q}^j} + \parder{L}{z}\parder{L}{\dot{q}^j} \right)\parder{}{\dot{q}^i} .
\end{align}

\section{Hybrid contact systems}\label{sec_hybrid}
In the previous section, we have related the (smooth) Herglotz principle for action-dependent Lagrangians with contact Hamiltonian and Lagrangian systems. In this section, we introduce a similar counterpart for the nonsmooth Herglotz principle, by means of introducing the notion of simple hybrid contact Hamiltonian and Lagrangian systems.

A mechanical system with impacts can be characterized by the space in which its dynamics evolve, the surface in which the impacts take place, a vector field whose integral curves are the trajectories of the system inside the dynamical space and a map describing the dynamics in the impact surface.

A \emph{simple hybrid system} \cite{Johnson1994, Westervelt2018} is a tuple $\mathscr{H}=(D, X, S, \Delta)$, where $D$ is a smooth manifold called the \myemph{domain}, $X$ is a smooth 
vector field on $D$, $S$ is an embedded submanifold of $D$ with co-dimension $1$ called the \myemph{switching surface}, and $\Delta:S\to D$ is a smooth embedding called the \myemph{impact map}. The dynamics of $\mathscr{H}$ are given by
\begin{equation}
  \begin{array}{ll}
    \dot{\sigma}(t)=X(\sigma(t)), & \text{if } \sigma^{-}(t)\notin{S}, \\ 
    \sigma^{+}(t)=\Delta(\sigma^{-}(t)),& \text{if }\sigma^-(t)\in{S}, 
  \end{array}
\end{equation}
where $\sigma:I\subset\mathbb{R}\to D$, and $\sigma^{-}$, $\sigma^{+}$ denote the states immediately before and after the times when $\sigma$ intersects ${S}$, namely $\sigma^{-}(t)\coloneqq \displaystyle{\lim_{\tau\to t^{-}}}x(\tau)$,\, $\sigma^{+}(t)\coloneqq \displaystyle{\lim_{\tau\to t^{+}}}x(\tau)$ are the left and right limits of the state trajectory $\sigma(t)$.

\begin{definition}
A simple hybrid system $\mathscr{H}_H=(D, X, S, \Delta)$ is called a \emph{simple hybrid contact Hamiltonian system} if $(D, \eta)$ is a contact manifold and $X$ is the contact Hamiltonian vector field associated with the Hamiltonian function $H$. The triple $(D, \eta, H)$ will be called the \emph{underlying contact Hamiltonian system}
\end{definition}

Consider a simple hybrid contact Hamiltonian system $\mathscr{H}_H=(D, X_H, S, \Delta)$ such that $D=T^\ast Q \times \R$, where $(Q, \langle \cdot, \cdot \rangle)$ is a Riemannian manifold. Suppose that the switching surface is of the form
$$S=\{(q,p,z)\in T^{*}Q\times \R: h(q)=0\},$$ 
where $\langle\langle \cdot, \cdot \rangle\rangle_q$ denotes the inner product induced by $\langle \cdot, \cdot \rangle$ on $T_q^\ast Q$.
 Additionally, assume\footnote{This is always the case for physically acceptable impacts.} that $\Delta$ fibres over $Q\times \R$, i.e., $\Delta(q, p, z)=\left(q, \Delta_p(q, p,z), z\right)$.
 Then, the impact map $\Delta$ satisfies the impact equations \eqref{eqs_impact_Hamiltonian_Herglotz} given by the Hamiltonian nonsmooth Herglotz principle if and only if
\begin{equation}
\begin{aligned} 
 & p_i^-  v^i 
  = p_i^+  v^i ,\\
  &H (q, p^-, z ) = H(q, p^+, z ),
\end{aligned}
\end{equation}
for any $v \in T_q (h^{-1}(0))$, where $\Delta(q, p^-, z)=(q, p^+, z)$.
If that is the case,  $\mathscr{H}_H$ is called a {\emph{simple hybrid Hamiltonian Herglotz system}}.

\begin{remark}
 It is worth noting that the dynamics of the simple hybrid Hamiltonian Herglotz system defined above coincide with the ones given by the Hamiltonian nonsmooth Herglotz principle \ref{Herglotz_Hamiltonian} for the Hamiltonian function $H$. That is, the integral curves of $X$ satisfy Eq.~\eqref{eqs_vector_Hamiltonian_Herglotz} and the impact map satisfies Eq.~\eqref{eqs_impact_Hamiltonian_Herglotz}. 
\end{remark}

\begin{definition}
A \emph{simple hybrid contact Lagrangian system} $\mathscr{H}_L=(D, X, S, \Delta)$ is a simple hybrid contact Hamiltonian system with underlying contact Hamiltonian system $(TQ\times \R, \eta_L, E_L)$, where $\eta_L$ and $E_L$ denote the contact form and the energy associated with a regular Lagrangian function $L$, respectively. In other words, the domain is of the form $D=TQ\times \R$, and the vector field $X$ is the SODE $\Gamma_L$ associated with the Lagrangian $L$.
\end{definition}

Suppose that $(Q, \langle \cdot, \cdot \rangle)$ is a Riemannian manifold and that the switching surface is of the form
$$S=\{(q,\dot q,z)\in T Q\times \R: h(q)=0\}.$$ 
 Assume that $\Delta$ fibres over $Q\times \R$.
 Then, the impact map $\Delta$ satisfies the impact equations \eqref{conditions_impact} given by the Nonsmooth Herglotz principle if and only if
\begin{subequations}
\begin{flalign}
  & \left.\frac{\partial L} {\partial \dot q^i}\right|_{(q, \dot q^-, z)}  v^i
  = \left.\frac{\partial L} {\partial \dot q^i}\right|_{(q, \dot q^+, z)} v^i, \\
  & E_L (q, \dot q^-, z)   = E_L (q, \dot q^+, z)
\end{flalign} 
\end{subequations}
for any $v \in T_q (h^{-1}(0))$, where $\Delta(q, \dot q^-, z)=(q, \dot q^+, z)$.
If that is the case,  $\mathscr{H}_L$ is called a {\emph{simple hybrid Lagrangian Herglotz system}}.

\begin{example}
The billiard with dissipation (see Section \ref{sec_billiard}) can be described by the simple hybrid Hamiltonian Herglotz system
$\mathscr{H}_H=(T^\ast \R^2 \times \R, X_H, S, \Delta)$, with
\begin{align}
 X_H 
&=  p_x \frac{\partial  } {\partial x} + p_y \frac{\partial  } {\partial y} - \gamma p_x \frac{\partial  } {\partial p_x} - \gamma p_y \frac{\partial  } {\partial p_y}\\
&+\left( \frac{1}{2} (p_x^2+p_y^2) - \gamma z\right) \frac{\partial  } {\partial z}
, \\
 S& = \left\{(x,y, p_x, p_y, z)\in T^\ast \R^2 \times \R \mid x^2+y^2 =1  \right\},
\end{align}
and $\Delta(x, y, p_x^-, p_y^-, z)=(x, y, p_x^+, p_y^+, z)$, where $p_x^+$ and $p_y^+$ are given by Eqs.~\eqref{impact_map_billiard}.


\end{example}

\section{Conclusions and Future work}\label{sec_conclusions}


In this paper we have developed a non-differentiable Herglotz variational principle that allows us to deal with impact problems. This principle is connected with the formulation of hybrid contact systems, and also takes into account the inherent dissipation of these systems in addition to the impact problem.

In subsequent work we shall study a number of issues related to those discussed in this paper:
\begin{itemize}
\item The reduction of this type of systems when there are symmetries that leave the Lagrangian invariant.
\item The construction of variational integrators that preserve the qualitative behaviour of the system.
\item Additionally, we want to prove a Carnot Theorem that accounts for the energy lost or gained in these types of situations.
\end{itemize}

\let\emph\oldemph
\bibliography{biblio}
\bibliographystyle{IEEEtran}
\let\oldemph\emph


\end{document}